\documentclass[12pt]{extarticle}
\usepackage{amsmath, amsthm, amssymb, hyperref, color}
\usepackage[shortlabels]{enumitem}
\usepackage{graphicx}
\usepackage[all]{xypic}
\usepackage{makecell}
\usepackage[final]{pdfpages}
\setboolean{@twoside}{false}
\usepackage{pdfpages}
\usepackage{caption}
\usepackage{subcaption}
\usepackage{scalefnt}
\usepackage{verbatim}
\oddsidemargin \evensidemargin
\newtheorem{theorem}{Theorem}
\numberwithin{theorem}{section}
\newtheorem{proposition}[theorem]{Proposition}
\newtheorem{lemma}[theorem]{Lemma}
\newtheorem{corollary}[theorem]{Corollary}

\newtheorem{remark}[theorem]{Remark}
\newtheorem{example}[theorem]{Example}

\usepackage{amsmath}
\usepackage{amssymb}
\usepackage{amsthm}
\usepackage{mathrsfs}
\usepackage{mathtools}
\usepackage{tikz}
\usepackage{cprotect}
\usepackage{float}

\def\C{{\mathbb C}}
\def\R{{\mathbb R}}

\def \P{{\mathbb P}}

\title{\bf Ranks and Symmetric Ranks \\ of Cubic Surfaces}

\author{Anna Seigal}
\date{}

\begin{document}

\maketitle

\begin{abstract}
We study cubic surfaces as symmetric tensors of format $4 \times 4 \times 4$. We consider the non-symmetric tensor rank and the symmetric Waring rank of cubic surfaces, and show that the two notions coincide over the complex numbers. 
The corresponding algebraic problem concerns border ranks.
We show that the non-symmetric border rank coincides with the symmetric border rank for cubic surfaces. As part of our analysis, we obtain minimal ideal generators for the symmetric analogue to the secant variety from the salmon conjecture. We also give a test for symmetric rank given by the non-vanishing of certain discriminants.
The results extend to order three tensors of all sizes, 
implying the equality of rank and symmetric rank when the symmetric rank is at most seven, and the equality of border rank and symmetric border rank when the symmetric border rank is at most five.
We also study real ranks via the real substitution method.
\end{abstract}

\section{Introduction}

A cubic surface is the zero set in $\P^3$ of a homogeneous cubic polynomial in four variables,
$$f = c_{3000} x_1^3 + c_{2100} x_1^2 x_2 + c_{1200} x_1 x_2^2 + c_{0300} x_2^3 + c_{2010} x_1^2 x_3 + \cdots + c_{0003} x_4^3 .$$
Such a polynomial has 20 coefficients,
hence the space of cubic surfaces is 19-dimensional.
Cubic surfaces are a central topic of study in classical algebraic geometry,
and a motivating example for more modern topics. Most prominently, 
the discovery of the 27 lines on the cubic surface in 1849 is celebrated as the beginning of modern algebraic geometry~\cite{S,V}.

We study cubic surfaces from the perspective of tensors, the multidimensional generalization of matrices.
A symmetric tensor $T$ of size $4 \times 4 \times 4$ has entries $T_{ijk}$, for $1 \leq i, j, k \leq 4$, which satisfy the symmetry relations $T_{ijk} = T_{ikj} = T_{jik} = T_{jki} = T_{kij} = T_{kji}$.
The space of $4 \times 4 \times 4$ symmetric tensors up to scale is also 19-dimensional.
Homogeneous quaternary cubics and symmetric $4 \times 4 \times 4$ tensors are in bijection via the correspondence
$$ f(x_1, x_2, x_3, x_4) = \sum_{i,j,k = 1}^4 T_{ijk} x_i x_j x_k .$$

More generally, homogeneous polynomials of degree $d$ in $n$ variables are in bijection with
symmetric tensors of size $n \times \cdots \times n$ ($d$ times).
It is a question of classical interest to find the shortest decomposition 
of a degree $d$ polynomial $f \in \mathbb{K}[x_1, \ldots, x_n]$
into a sum of powers,
$f = \sum_{i = 1}^r \lambda_i l_i^d$, where each $l_i \in \mathbb{K}[x_1, \ldots, x_n]$ is a linear form and $\lambda_i \in \mathbb{K}$. The minimal number of summands in such a decomposition is the {\em Waring rank} of $f$ over $\mathbb{K}$.
Equivalently, the Waring rank is the {\em symmetric rank} of the tensor corresponding to $f$, the length $r$ of its shortest decomposition as a sum of symmetric rank one tensors, $T = \sum_{i = 1}^r \lambda_i u_i^{\otimes d}$.
The vector $u_i \in \mathbb{K}^n$ lists the coefficients of the linear form~$l_i$.
There is also the notion of {\em (non-symmetric) rank} over $\mathbb{K}$. This is the length of the shortest decomposition of a tensor into a sum of rank one tensors, $T = \sum_{i = 1}^r u_i \otimes v_i \otimes \cdots \otimes w_i$, where $u_i, v_i, \ldots, w_i$ are $d$ vectors in $\mathbb{K}^n$. In the context of tensors the degree $d$ is called the order.

The set of tensors of rank $\leq r$ is not closed whenever it is a proper subset of the space of tensors and $r > 1$. The same holds for the space of symmetric tensors of symmetric rank~$\leq r$~\cite{CGLM}. In light of this, each of the above notions of rank have closed analogues called {\em border ranks}. For a given notion of rank, and tensor $T$, the border rank is the smallest $r$ such that~$T$ lies in the closure of the rank $r$ tensors. 

The rank of a tensor depends on the field $\mathbb{K}$ over which the decomposition is taken.
In this article, we focus on the case of ranks defined over the complex numbers, $\mathbb{K} = \C$, and also consider the real rank case $\mathbb{K} = \R$. When the field is not specified, we are referring to the usual complex case.
The real and complex rank of a tensor need not agree. 
It follows from the definitions that complex rank is bounded above by real rank, border rank is less than or equal to rank, and non-symmetric rank is less than or equal to symmetric rank. For matrices the real and complex ranks agree, and also rank and border rank are the same. This is not true for tensors, as we see in the following two examples.

\begin{example}[The cubic surface $x_1^2 x_2 + x_3^3 - x_3 x_4^2$] 
The monomial $x_1^2 x_2$ has (complex or real) border rank two, and (complex or real) rank three. Evaluating the hyperdeterminant of $x_3^3 - x_3 x_4^2$ shows that it has complex rank two and real rank three~\cite{SS}, and its rank and border rank are equal.
Since the variables from the two parts of the sum are disjoint, and Strassen's Conjecture~\cite[\S5.7]{L} holds here, the cubic surface has complex border rank four, real border rank five, complex rank five and real rank six.
\end{example}

\begin{example}[The cubic threefold $x_1(x_1 x_2 + x_3^2 ) + x_4^2 - x_4 x_5^2$]
The cubic curve $x_1(x_1 x_2 + x_3^2)$ is a conic and tangent line. It has 
(real or complex) border rank three and (real or complex) rank five~\cite{B}. Hence, using the previous example, the cubic threefold has complex border rank five, real border rank six, complex rank seven and real rank eight.
\end{example}

Tensors and symmetric tensors arise in applications such as complexity theory, multivariate statistics, medical imaging, multiway factor analysis, numerical analysis, and signal processing (see \cite{CGLM,L2,L} and references therein). Just as for a matrix, the rank of a tensor is one of its fundamental properties. For a tensor of data, the rank is the number of `signals' which are combined in the data.
A key result from linear algebra says that a symmetric matrix of rank~$r$ can be written as a sum of $r$ rank one {\em symmetric} matrices. The generalization of this result to tensors is a topic of ongoing study.

Comon's conjecture states the equality of the rank and symmetric rank of a symmetric tensor. First posed 10 years ago~\cite{CGLM}, it has been conjectured for {\em complex rank}, {\em complex border rank}, {\em real rank} and {\em real border rank}. The conjecture has been proved in many special cases: when the symmetric rank is at most two~\cite{CGLM}, when the rank is less than or equal to the order~\cite{ZHQ}, and when the rank is at most the flattening rank plus one~\cite{F}. 
Furthermore, the conjecture has been proved to generically hold in certain families of tensors~\cite{BBCG}. 
On the other hand, a counter-example to Comon's conjecture for complex rank has been found by Shitov~\cite{Sh}, a tensor of size $800 \times 800 \times 800$. It is an open problem to characterize which tensors have the same rank and symmetric rank, for different notions of rank.

\bigskip

In this article, we study the smallest tensor
  format for which the agreement of rank and symmetric rank was not known: cubic surfaces, or symmetric $4 \times 4 \times 4$ tensors.
There does not exist a finite list of normal forms in this case, because the dimension of the general linear group $PGL_4$ is~15, whereas the space of cubic surfaces is 19-dimensional. 
  We prove the following result.
  
    \begin{theorem} \label{rank} The rank and symmetric rank agree for cubic surfaces. \end{theorem}

The conclusion extends to arbitrary tensors of format $n \times n \times n$, by giving larger ranges of ranks among which all tensors have agreement of rank and symmetric rank.
 
  \begin{corollary}\label{corrank}The rank and symmetric rank of a cubic polynomial in $n$ variables (order three symmetric tensor) are the same, whenever the symmetric rank is at most seven.
 \end{corollary}
 
We make the following contributions for border ranks over the complex numbers.
 
 \begin{theorem}\label{border}The border rank and symmetric border rank agree for cubic surfaces.
 \end{theorem}
 
  \begin{corollary}\label{corbor} The rank and symmetric border rank of a cubic polynomial are the same whenever the symmetric border rank is at most five.
 \end{corollary}
 
We also consider ranks over the real numbers. We show that real rank and real border rank agree for generic cubic surfaces, and we study special cases in greater detail.

\bigskip

The notion of {\em flattening rank} is useful in our study. The flattening ranks of a tensor (see e.g.~\cite{L}) are the ranks of its flattening matrices, the reshapings of its entries into matrix format. 
General tensors have a tuple of flattening ranks, one for each distinct flattening. In the $n \times n \times n$ symmetric case the flattening rank is a single number, the rank of the $n \times n^2$ flattening matrix. Since the flattening ranks are ranks of matrices, they inherit properties possessed by matrix rank (such as being closed, equivalence of real and complex rank, and equivalence of non-symmetric and symmetric rank). 
From the definition, the flattening rank cannot exceed any of the ranks described above; the flattening rank of $T = \sum_{i = 1}^r v_i^{\otimes d}$ is the dimension of the span of the vectors $\{ v_1, \ldots, v_r \}$, which is less than or equal to~$r$.

The space of symmetric $n \times n \times \cdots \times n$ ($d$ times) tensors with entries in a field $\mathbb{K}$ is denoted $S^d(\mathbb{K}^n)$. The analogous space of (not necessarily symmetric) tensors is $\mathbb{K}^n \otimes \cdots \otimes \mathbb{K}^n$ ($d$ times).

\bigskip 

The rest of this article is organized as follows. We prove the results in Theorem~\ref{rank} and Corollary~\ref{corrank} concerning complex rank in Section~\ref{cpxrnk}. We prove the complex border rank results in Theorem~\ref{border} and Corollary~\ref{corbor} in Section~\ref{cpxborder}. We discuss real ranks in Section~\ref{realranks}.

\section{Ranks of cubic surfaces} \label{cpxrnk}

In this section we prove Theorem~\ref{rank}, that the rank and symmetric rank coincide for a cubic surface over the complex numbers. We use the following three results, as well as the classification of non-singular cubic surfaces in~\cite{S}.

\begin{theorem}[Sylvester's Pentahedral Theorem (1851), see~{\cite[\S84]{S}}] \label{sylv}
A generic cubic surface
can be decomposed uniquely as the sum of five cubes of linear forms, $f = l_1^3 + l_2^3 + l_3^3 + l_4^3 + l_5^3$ where each $l_i \in \C[x_1,x_2,x_3,x_4]$ is a linear form.
\end{theorem}

\begin{theorem}[{\cite[Theorem 1.1]{F}}] \label{friedland}
Let $\mathbb{K}$ be a field with at least three elements. Consider a tensor $T \in S^d(\mathbb{K}^n)$ whose rank is bounded above by its flattening rank plus one. Then the rank and symmetric rank of $T$ defined over $\mathbb{K}$ coincide. \end{theorem}

It follows from Theorem~\ref{friedland} that, if the symmetric rank is bounded above by the flattening rank plus two, then the rank and symmetric rank coincide: the  alternative is that the rank is strictly less than the symmetric rank, which means it satisfies the hypothesis of the theorem.

\begin{theorem}[The substitution method, e.g. {\cite[\S5.3.1]{L2}}] \label{substitution}
Let $T \in \C^{n_1} \otimes \C^{n_2}  \otimes \C^{n_3}$ be a tensor of rank $r$. We write $T = \sum_{i = 1}^{n_1} e_i \otimes M_i$, where $\{ e_i : 1 \leq i \leq {n_1} \}$ are the elementary basis vectors, and the $M_i$ are $n_2 \times n_3$ matrices, known as the slices of the tensor. Reordering indices to ensure that $M_{n_1} \neq 0$, there exist constants $\lambda_1, \ldots, \lambda_{n_1-1}$ such that the following $(n_1-1) \times n_2 \times n_3$ tensor has rank at most~$r-1$:
$$ \sum_{i=1}^{n_1-1} e_i \otimes (M_i - \lambda_i M_{n_1}) .$$
If the matrix $M_{n_1}$ has rank one, the tensor above has rank exactly $r-1$.
\end{theorem}

In the following four subsections we prove equality of rank and symmetric rank for the family of cubic surfaces in the title. Together the subsections prove Theorem~\ref{rank}. 

\subsection{Cones over cubic curves}

Cones over cubic curves have a natural characterization in terms of tensors: they have sub-generic flattening rank, and parametrize the {\em subspace variety}~\cite[\S7.1]{L} defined by the vanishing of the $4 \times 4$ minors of the flattening. In this section we prove Theorem~\ref{rank} for cubic surfaces with sub-generic flattening rank.
The defining polynomials of such cubic surfaces have a change of basis that removes one of the four variables. 
We change coordinates by an element $M$ of the general linear group $GL_4$ to obtain a tensor $T'$ with non-zero entries only in its upper-left $3 \times 3 \times 3$ block. Its entries are expressed in terms of $T$ and $M$ as
$$ T_{ijk}'= \sum_{a,b,c = 1}^4 T_{abc} M_{ai} M_{bj} M_{ck}  .$$
Rank is invariant under general linear group action, hence $T'$ has the same rank as $T$. Given an expression for $T'$ as a sum of rank one tensors, setting the fourth entry of all vectors that appear in the decomposition to zero gives a valid expression with the same number of terms.

Hence, to study ranks of cones over cubic curves it suffices to study ranks of plane cubic curves as symmetric $3 \times 3 \times 3$  tensors. 
It is known that the rank and symmetric rank agree for cubic curves of sub-generic flattening rank ($2 \times 2 \times 2$ tensors and rank one tensors) e.g. via their normal forms. Cubic curves have a generic flattening rank of three. Theorem~\ref{friedland} says that the rank and symmetric rank coincide provided that the symmetric rank is at most five. The classification of cubic curves 
in~\cite[\S96]{S} shows that five is the maximum possible symmetric rank. This concludes the proof for cones over cubic curves.

\subsection{Non-singular cubic surfaces} \label{non-singular}

Based on the previous subsection, it remains to consider cubic surfaces with flattening rank four. When the rank is at most five, Theorem~\ref{friedland} implies that the rank and symmetric rank coincide.
This leaves the cubic surfaces that are not expressible as a sum of five linear powers, those for which Theorem~\ref{sylv} fails to give a decomposition.
There are two such families of non-singular cubic surfaces, see~\cite[\S94]{S}, with equations
\begin{equation}\label{2.2} \begin{matrix} (x_1^3 + x_2^3 + x_3^3 ) + x_4^2 ( \lambda_1 x_1 +  \lambda_2 x_2 +  \lambda_3 x _3 + \lambda_4 x_4 ) , \\
 \mu_1 x_1^3 + x_2^3 + x_3^3 - 3 x_1 ( \mu_2 x_1 x_2 + x_1 x_3 + x_4^2 ) . \end{matrix} \end{equation}
The parameters $\lambda_i$, $\mu_j$ are arbitrary subject to maintaining non-singularity. The failure of Sylvester's Pentahedral Theorem for these surfaces is due to the non-genericity of their Hessian quartic surface, which has fewer than 10 distinct singular points. 
These cubics have symmetric rank six~\cite[\S97]{S}. The non-symmetric rank cannot be five or less by Theorem~\ref{friedland}.

\subsection{Cubic surfaces with infinitely many singular points}

We begin with the reducible cubic surfaces, followed by the irreducible cubic surfaces with infinitely many singular points. The three normal forms of reducible cubic surfaces are given in~\cite{CGV}. They are $x_1 ( x_1^2 + x_2^2 + x_3^2 + x_4^2)$, $x_1 ( x_2^2 + x_3^2 + x_4^2)$, and $x_1 ( x_1 x_2 + x_3^2 + x_4^2)$. 
The first two have symmetric rank six~\cite{CGV}, hence by Theorem~\ref{friedland} they also have rank six. The third has symmetric rank seven~\cite{S}. We show that the rank of this normal form is seven, and hence that its rank and symmetric rank agree.

\begin{proposition}\label{rank7}
The cubic surface $ f = x_1 ( x_1 x_2 + x_3^2 + x_4^2 )$ has non-symmetric rank seven.
\end{proposition}

\begin{proof}
The polynomial $f$ can be written up to scale as the symmetric $4 \times 4 \times 4$ tensor
$$ e_1 \otimes 
\begin{bmatrix}
0 & 1 & 0 & 0 \\ 
1 & 0 & 0 & 0 \\ 
0 & 0 & 1 & 0 \\ 
0 & 0 & 0 & 1 \\ 
\end{bmatrix}
+ e_2 \otimes
\begin{bmatrix}
1 & 0 & 0 & 0 \\ 
0 & 0 & 0 & 0 \\ 
0 & 0 & 0 & 0 \\ 
0 & 0 & 0 & 0 \\ 
\end{bmatrix}
+ e_3 \otimes
\begin{bmatrix}
0 & 0 & 1 & 0 \\ 
0 & 0 & 0 & 0 \\ 
1 & 0 & 0 & 0 \\ 
0 & 0 & 0 & 0 \\ 
\end{bmatrix}
+ e_4 \otimes
\begin{bmatrix}
0 & 0 & 0 & 1 \\
0 & 0 & 0 & 0 \\
0 & 0 & 0 & 0 \\
1 & 0 & 0 & 0 \\
\end{bmatrix} . $$
We apply Theorem~\ref{substitution} iteratively to the second, third, and fourth slices of $f$. The slices are linearly independent $4 \times 4$ matrices. No linear combination of them can be subtracted from the first slice to give a vanishing determinant. These two observations imply that the rank of~$f$ is bounded from below by $4 + 3 = 7$. Since the symmetric rank is seven, the non-symmetric rank cannot exceed seven.
\end{proof}

There are two normal forms of irreducible cubic surfaces with infinitely many singular points~\cite[\S97]{S}, with representatives 
$ x_1 x_2^2 + x_3 x_4^2$,
which has symmetric rank six, and
$ x_1^2 x_2 + x_1 x_3 x_4 + x_3^3 $ with symmetric rank at most seven.
In the former case the non-symmetric rank is also six, using Theorem~\ref{friedland}. In the latter case we follow an approach as in Proposition~\ref{rank7}.

\begin{proposition} \label{rank7no2}
The cubic surface $ x_1^2 x_2 + x_1 x_3 x_4 + x_3^3 $ has non-symmetric rank seven.
\end{proposition}

\begin{proof}
The polynomial $f$ is the symmetric $4 \times 4 \times 4$ tensor
$$ e_1 \otimes 
\begin{bmatrix}
0 & \frac13 & 0 & 0 \\ 
\frac13 & 0 & 0 & 0 \\ 
0 & 0 & 0 & \frac16 \\ 
0 & 0 & \frac16 & 0 \\ 
\end{bmatrix}
+ e_2 \otimes
\begin{bmatrix}
\frac13 & 0 & 0 & 0 \\ 
0 & 0 & 0 & 0 \\ 
0 & 0 & 0 & 0 \\ 
0 & 0 & 0 & 0 \\ 
\end{bmatrix}
+ e_3 \otimes
\begin{bmatrix}
0 & 0 & 0 & \frac16 \\ 
0 & 0 & 0 & 0 \\ 
0 & 0 & 1 & 0 \\ 
\frac16 & 0 & 0 & 0 \\ 
\end{bmatrix}
+ e_4 \otimes
\begin{bmatrix}
0 & 0 & \frac16 & 0 \\
0 & 0 & 0 & 0 \\
\frac16 & 0 & 0 & 0 \\
0 & 0 & 0 & 0 \\
\end{bmatrix} . $$
The second, third, and fourth slices are linearly independent. No linear combination of them can be subtracted from the first slice to give a vanishing determinant. Hence the rank is at least $4+3=7$. The symmetric rank is at most seven hence both ranks are seven.
\end{proof}

\subsection{Cubic surfaces with finitely many singular points}

We introduce a test to show that a cubic surface $f$ has symmetric rank at most five. The test checks that $f$ does not lie on two discriminant loci which contain the tensors of higher rank.
We use it to prove Theorem~\ref{rank} for cubic surfaces with finitely many singular points.

The singular cubic surfaces lie on the discriminant hypersurface~\cite{St}.
Non-singular cubic surface, on the complement of the hypersurface, have symmetric rank at most five unless they are of the form in equation~\eqref{2.2}. The surfaces in~\eqref{2.2} are contained in a second discriminant locus, which we describe. Our test is the following:  if neither discriminant vanishes at $f$, it has symmetric rank at most five.

We now explain how to construct the second discriminant. The determinant of a $4 \times 4$ symmetric matrix of indeterminates defines a hypersurface with 10 singular points, where the $3 \times 3$ minors vanish. The $4 \times 4$ symmetric matrices of linear forms whose determinant hypersurface has fewer than 10 singular points lie on the Hurwitz form of the variety of rank two $4 \times 4$ symmetric matrices~\cite{St2}. Applying~\cite[Theorem 1.1]{St2} shows that the Hurwitz form in this setting is an irreducible hypersurface of degree $30$ in the Pl\"ucker coordinates of codimension six linear spaces, since the sectional genus is six. The Hurwitz form has degree 120 in the coordinates of the indeterminates, since each Pl\"ucker coordinate has degree four.

The Hessian matrix of a cubic surface is a $4 \times 4$ symmetric matrix of linear forms, the second order partial derivatives. The determinant of the matrix is the defining equation of the Hessian surface, which generically has 10 singular points at which the $3 \times 3$ minors of the matrix vanish. 
The cubic surfaces in~\eqref{2.2} are special in that their Hessian surfaces have fewer than 10 distinct singular points. 
Hence they lie on the specialization
of the Hurwitz form above to Hessian matrices of cubic surfaces.
This is a discriminant hypersurface in the space of cubic surfaces, which we call the {\em Hessian discriminant}. It divides the specialization of the Hurwitz form. The above paragraph implies the following.

\begin{proposition}
The Hessian discriminant is a hypersurface of degree at most 120 in the 20 coefficients of the cubic surfaces.
\end{proposition}

We obtain a test for a cubic surface having symmetric rank at most six as follows.
If there exists a linear form $l$ such that $f + l^3$ has symmetric rank at most five, then $f$ has symmetric rank at most six. 
To check that $f$ has symmetric rank at most six, it suffices to check that neither discriminant vanishes identically on the set of cubic surfaces of the form $f + l^3$, as $l$ ranges over the linear forms.
We first prove this for the discriminant of singular cubics via the following result, which is stated without proof in~\cite[\S97]{S}.

\begin{lemma}\label{finitely}
Let $f \in \C[x_1, x_2, x_3, x_4]$ be a cubic surface with finitely many singular points. Then for a generic linear form $l \in \C[x_1, x_2, x_3, x_4]$ the cubic surface $f + l^3$ is non-singular.
\end{lemma}

\begin{proof}
A generic $l$ satisfies $l(p) \neq 0$ at all singular points of $f$, since the plane perpendicular to the coefficients of $l$ needs to avoid finitely many points. A singular point of $g = f + l^3$ at which $l(p) \neq 0$ must satisfy $g(p) = 0$, and
$ \left( \frac{\partial f}{\partial x_1}|_p : \frac{\partial f}{\partial x_2}|_p : \frac{\partial f}{\partial x_3}|_p : \frac{\partial f}{\partial x_4}|_p \right) = (l_1 : l_2 : l_3 : l_4 ) $.
The partial derivatives of $f$ as $p$ varies over $g = 0$ parametrize a subset of $\P^3$ of dimension at most two. Hence for generic~$l$ this equation will not be satisfied at any $p$ on the surface~$g$.
\end{proof}

\begin{remark}
Lemma~\ref{finitely} can fail for surfaces with infinitely many singular points, such as $x_1(x_1 x_2 + x_3^2 + x_4^2)$ from Proposition~\ref{rank7}. It is singular at $(x_1 : x_2 : x_3 : x_4) = (0 : t_1 : t_2 : \pm i t_2 )$ for $(t_1 : t_2) \in \P^1$. Every linear form $l$ vanishes at a non-zero singular point of $f$ and at that point $f + l^3$ is also singular. 
\end{remark}

We now prove the following result concerning the Hessian discriminant, which uses computations in the computer algebra systems Macaulay2, Magma and Maple.

\begin{lemma} \label{disc2}
For all cubic surfaces with finitely many singular points, except those of singularity type $E_6$, there exists a linear form $l$ such that $f + l^3$ does not lie on the Hessian discriminant.
\end{lemma}

\begin{proof}
We refer to the classification of cubic surfaces with finitely many singular points in~\cite{BW,Sch,singsurf}. There are infinitely many normal forms, which fall into~20 classes according to the structure of the singularities. Thirteen classes have a single normal form representative. For these, we compute in Macaulay2 the ideal of singular points of the Hessian of $f + l^3$ for random linear form $l$. For 12 classes, all except singularity type $E_6$, this computation gives an ideal of degree 10 and $f+l^3$ does not lie on the Hessian discriminant.

It remains to consider the seven classes from~\cite[Theorem 2]{Sch} which are given in terms of parameters, $f = f(\rho)$. We sample linear forms $l_i$ and compute the discriminant of $f(\rho) + l_i^3$. This gives a polynomial condition in the parameters which vanishes when $f(\rho) + l_i^3$ lies on the Hessian discriminant. We consider sufficiently many linear forms, in order that there does not exist a choice of parameters such that the Hessian discriminant vanishes at $f(\rho) + l_i^3$ for all linear forms in the sample. We choose linear forms for which the computation to form the discriminant is not prohibitively slow. 
We construct the discriminant using Macaulay2 or Maple, and check that no parameters satisfy all discriminants using Macaulay2 or Magma. 

Often a good choice of linear form is $l= 0$; if the Hessian discriminant does not vanish at $f$ then it also does not vanish at $f + l^3$ when $l$ has sufficiently small coefficients. 
In some cases we consider enough linear forms such that the Hessian discriminant only vanishes at all $f(\rho) + l_i^3$ for a finite number of parameters~$\rho$, and then we check the remaining parameter values one by one in the same way as for the single normal form representatives.
\end{proof}

It remains to consider the singularity type $E_6$, with normal form $x_1^2 x_4  + x_1 x_3^2 + x_2^3$.
Here we show that the symmetric rank is at most six directly, since the normal form can be re-written as a sum of six linear powers,
$$ \frac16 x_4^3 + \frac16 (2 x_1 + x_4)^3 - \frac13 ( x_1 + x_4)^3 + x_2^3 - \frac12 ( x_1 + \frac{i}{\sqrt{3}} x_3)^3 - \frac12 ( x_1 - \frac{i}{\sqrt{3}} x_3)^3 .$$ 
Hence we have proved the following.

\begin{theorem}
Cubic surfaces with finitely many singular points have symmetric rank at most six.
\end{theorem}

When the symmetric rank is at most six, the equality of rank and symmetric rank follows from Theorem~\ref{friedland}, hence this concludes our proof of Theorem~\ref{rank}. 
To conclude the section we prove Corollary~\ref{corrank}.

\begin{proof}[Proof of Corollary~\ref{corrank}]
By Theorem~\ref{rank}, it remains to consider tensors of flattening rank five or more. By Theorem~\ref{friedland}, the rank and symmetric rank agree when the rank is at most the flattening rank plus one. Hence they agree up to rank six, and symmetric rank seven.
\end{proof}

\section{Border ranks of cubic surfaces} \label{cpxborder}

The set of rank one $n \times n \times n$ tensors and the set of rank one $n \times n \times n$ symmetric tensors, up to scale, are respectively the Segre and Veronese varieties in complex projective space. We denote them by
$$ S_n := {\rm Seg} (\P^{n-1} \times \P^{n-1} \times \P^{n-1}) \qquad \text{and} \qquad V_n := \nu_3 (\P^{n-1}).$$
The $r$th secant variety $\sigma_r(S_n)$ consists of all tensors of non-symmetric border rank at most~$r$. Likewise $\sigma_r(V_n)$ consists of all tensors of symmetric border rank at most $r$~\cite{L}. The linear subspace of symmetric tensors inside $\C^n \otimes \C^n \otimes \C^n$ is denoted $L_n$.

We prove Theorem~\ref{border}, that the border rank and symmetric border rank agree for cubic surfaces, by establishing the equality of 
$\sigma_r(V_4)$ and $\sigma_r(S_4) \cap L_4$ for all~$r$.
For symmetric $n \times n \times n$ tensors, we prove Corollary~\ref{corbor}, that the border rank and symmetric border rank agree up to symmetric border rank five, by showing that 
$ \sigma_r(V_n) = \sigma_r(S_n) \cap L_n$ for all $n$ and whenever $r \leq 4$.

\subsection{Border ranks of cones over cubic curves}

As for the rank result, we begin by considering cones over cubic curves. For such surfaces, we can apply a symmetric change of basis to ensure that only the top-left $3 \times 3 \times 3$ block contains non-zero entries. The tensors in any approximating sequence can always be chosen to have this property, hence it suffices to consider cubic curves.
The space of cubic curves is 10-dimensional. The secant varieties of the Veronese variety $V_3 = \nu_3 ( \P^2)$ are not defective, by the Alexander-Hirschowitz Theorem~\cite{AH}. The dimensions are 
$$ {\rm dim}(V_3) = 2, \quad \rm{dim}(\sigma_2(V_3)) = 5 \quad {\rm dim}(\sigma_3(V_3)) =8, \quad \rm{dim}(\sigma_4(V_3)) = 10.$$
Since the fourth secant variety fills the space $S^3(\C^3)$, cubic curves have border rank $\leq 4$.

\begin{lemma}\label{curves}
The border rank and symmetric border rank of cubic curves coincide.
\end{lemma}

\begin{proof}
We compare the equations defining the secant variety $\sigma_r(V_3)$ with the symmetric restriction of the equations defining the non-symmetric secant $\sigma_r(S_3)$, for $1 \leq r \leq 4$.
The equations defining the Segre variety $S_3$ are the $2 \times 2$ minors of all flattenings. Restricting these equations to symmetric tensors gives the equations defining $V_3$, the $2 \times 2$ minors of the most symmetric catalecticant. Similarly $\sigma_2(S_3)$ is given by the vanishing of the $3 \times 3$ minors of the flattenings. Restricting to symmetric tensors, we get the equations for $\sigma_2(V_3)$, the $3 \times 3$ minors of the most symmetric catalecticant. The equations defining $\sigma_3 (S_3)$ are Strassen's commuting conditions. Restricting these to symmetric tensors recovers the Aronhold invariant which defines $\sigma_3(V_3)$, see~\cite[Exercise 3.10.1.2]{L}.

Cubic curves outside $\sigma_3(V_3)$ have non-symmetric border rank at least four, as they do not lie in the symmetric restriction of $\sigma_3(S_3)$. Their non-symmetric border rank cannot exceed their symmetric border rank, so the non-symmetric border rank must be exactly four. \end{proof}

\subsection{Symmetric salmon equations}

Finding ideal generators for the secant variety $\sigma_4 (S_4)$ is the salmon conjecture.
In~\cite{BO,FG}, set-theoretic equations for the variety are found, although ideal-theoretic equations are not known. Here we obtain the prime ideal for $\sigma_4(V_4)$, a {`symmetric salmon'} result.

The description for the set $\sigma_4(S_4)$ consists of equations in degrees five, six and nine. The degree five equations make a 1728-dimensional module. Restricting the equations in this module to symmetric tensors yields 36 linearly independent quintics in the coefficients of the cubic surfaces which vanish on the set $\sigma_4(V_4)$. One of the quintics is
\begin{tiny} $$ \begin{matrix} 
16c_{1002}^2c_{0201}c_{0120}^2
-8c_{1002}^2c_{0210}c_{0120}c_{0111}
-12c_{1011}c_{1002}c_{0201}c_{0120}c_{0111} 
+4c_{1011}c_{1002}c_{0210}c_{0111}^2
+c_{1011}^2c_{0201}c_{0111}^2 \\
+4c_{1020}c_{1002}c_{0201}c_{0111}^2+4c_{1101}c_{1002}c_{0120}c_{0111}^2-c_{1101}c_{1011}c_{0111}^3-2c_{1110}c_{1002}c_{0111}^3+c_{2001}c_{0111}^4 \\
+8c_{1011}c_{1002}c_{0210}c_{0120}c_{0102} +4c_{1011}^2c_{0201}c_{0120}c_{0102}-16c_{1101}c_{1002}c_{0120}^2c_{0102} -4c_{1011}^2c_{0210}c_{0111}c_{0102} \\ -8c_{1020}c_{1002}c_{0210}c_{0111}c_{0102}-4c_{1020}c_{1011}c_{0201}c_{0111}c_{0102} 
+4c_{1101}c_{1011}c_{0120}c_{0111}c_{0102}+8c_{1110}c_{1002}c_{0120}c_{0111}c_{0102} \\
+2c_{1110}c_{1011}c_{0111}^2c_{0102} -8c_{2001}c_{0120}c_{0111}^2c_{0102}+8c_{1020}c_{1011}c_{0210}c_{0102}^2 
-8c_{1110}c_{1011}c_{0120}c_{0102}^2+16c_{2001}c_{0120}^2c_{0102}^2 \\ +16c_{1002}^2c_{0210}^2c_{0021}+
8c_{1011}c_{1002}c_{0210}c_{0201}c_{0021}-4c_{1011}^2c_{0201}^2c_{0021}-16c_{1020}c_{1002}c_{0201}^2c_{0021} 
+8c_{1101}c_{1002}c_{0201}c_{0120}c_{0021} \\ -12c_{1101}c_{1002}c_{0210}c_{0111}c_{0021} +4c_{1101}c_{1011}c_{0201}c_{0111}c_{0021}+8c_{1110}c_{1002}c_{0201}c_{0111}c_{0021}+c_{1101}^2c_{0111}^2c_{0021} \\
+4c_{1200}c_{1002}c_{0111}^2c_{0021}-8c_{2001}c_{0201}c_{0111}^2c_{0021}-4c_{1101}c_{1011}c_{0210}c_{0102}c_{0021}-16c_{1110}c_{1002}c_{0210}c_{0102}c_{0021} \\ +8c_{1101}c_{1020}c_{0201}c_{0102}c_{0021}
+16c_{1200}c_{1002}c_{0120}c_{0102}c_{0021}-16c_{2001}c_{0201}c_{0120}c_{0102}c_{0021} \\ -4c_{1110}c_{1101}c_{0111}c_{0102}c_{0021}-4c_{1200}c_{1011}c_{0111}c_{0102}c_{0021} 
+24c_{2001}c_{0210}c_{0111}c_{0102}c_{0021} \\ +4c_{1110}^2c_{0102}^2c_{0021}-16c_{1200}c_{1020}c_{0102}^2c_{0021}-4c_{1101}^2c_{0201}c_{0021}^2-16c_{1200}c_{1002}c_{0201}c_{0021}^2+16c_{2001}c_{0201}^2c_{0021}^2 \\
+8c_{1200}c_{1101}c_{0102}c_{0021}^2-16c_{1011}c_{1002}c_{0210}^2c_{0012}+16c_{1020}c_{1002}c_{0210}c_{0201}c_{0012}+8c_{1020}c_{1011}c_{0201}^2c_{0012} \\ +8c_{1101}c_{1002}c_{0210}c_{0120}c_{0012}
-4c_{1101}c_{1011}c_{0201}c_{0120}c_{0012}-16c_{1110}c_{1002}c_{0201}c_{0120}c_{0012}+4c_{1101}c_{1011}c_{0210}c_{0111}c_{0012} \\ +8c_{1110}c_{1002}c_{0210}c_{0111}c_{0012}
-4c_{1101}c_{1020}c_{0201}c_{0111}c_{0012}-4c_{1110}c_{1011}c_{0201}c_{0111}c_{0012}-4c_{1101}^2c_{0120}c_{0111}c_{0012} \\ -8c_{1200}c_{1002}c_{0120}c_{0111}c_{0012} 
+24c_{2001}c_{0201}c_{0120}c_{0111}c_{0012}+2c_{1110}c_{1101}c_{0111}^2c_{0012}-8c_{2001}c_{0210}c_{0111}^2c_{0012} \\ -8c_{1101}c_{1020}c_{0210}c_{0102}c_{0012} 
+8c_{1110}c_{1011}c_{0210}c_{0102}c_{0012}+8c_{1110}c_{1101}c_{0120}c_{0102}c_{0012}-8c_{1200}c_{1011}c_{0120}c_{0102}c_{0012} \\ -16c_{2001}c_{0210}c_{0120}c_{0102}c_{0012}
-4c_{1110}^2c_{0111}c_{0102}c_{0012}+16c_{1200}c_{1020}c_{0111}c_{0102}c_{0012}+4c_{1101}^2c_{0210}c_{0021}c_{0012} \\ +8c_{1200}c_{1011}c_{0201}c_{0021}c_{0012} 
-16c_{2001}c_{0210}c_{0201}c_{0021}c_{0012}-4c_{1200}c_{1101}c_{0111}c_{0021}c_{0012} \\ -8c_{1110}c_{1101}c_{0210}c_{0012}^2+16c_{2001}c_{0210}^2c_{0012}^2+4c_{1110}^2c_{0201}c_{0012}^2 
-16c_{1200}c_{1020}c_{0201}c_{0012}^2+8c_{1200}c_{1101}c_{0120}c_{0012}^2. \end{matrix} $$ \end{tiny}

\begin{proposition} \label{salmon} The prime ideal of $\sigma_4(V_4)$ is generated by 36 quintics.
\end{proposition}

\begin{proof}
The 36 quintics are obtained by restricting the degree five salmon equations to symmetric tensors. Using symbolic computations
 in Macaulay2, they are shown to generate an ideal of degree at most 105 and codimension 4. Their ideal is Gorenstein, with symmetric minimal free resolution
$$ R^1 \leftarrow R^{36} \leftarrow R^{70} \leftarrow R^{36} \leftarrow R^1 \leftarrow 0 ,$$
where $R=\C[x_1,x_2,x_3,x_4]$.
Using the numerical methods of Bertini, 
the highest dimensional component of the variety defined by the 36 quintics is shown to be irreducible, and to have degree 105. The Gorenstein property means the unmixedness theorem applies: there cannot be lower-dimensional components. The zero set of the 36 quintics contains the codimension four set $\sigma_4(V_4)$ of symmetric border rank four tensors, and hence since the codimensions agree, and the former set is irreducible, they are equal as sets. Furthermore, the ideal generated by the 36 quintics is prime, hence they generate the ideal of $\sigma_4(V_4)$.
\end{proof}

\begin{proposition}
The 36 quintics defining $\sigma_4(V_4)$ are the irreducible module $S_{5,4,4,2}(\C^4)$. 
\end{proposition}

\begin{proof}
Proposition~\ref{salmon} shows that $\sigma_4(V_4)$ is generated by 36 quintics. Since $\sigma_4(V_4)$ is invariant under $GL_4$ action, the quintics are a $GL_4$ module in the 42504-dimensional space of quintic polynomials in the coefficients of cubic surfaces, $S^5 (S^3 \C^4)$. The $GL_4$ modules in $S^5 (S^3 \C^4)$ are a subset of those from $(\C^4)^{\otimes 15}$. The irreducible modules of the latter are indexed by Young diagrams with 15 boxes and no more than four rows~\cite{L}.
We compute in SAGE which $GL_4$-modules from 
$(\C^4)^{\otimes 15}$ occur in the decomposition of $S^5 (S^3 \C^4)$, by evaluating
\begin{verbatim}
s = SymmetricFunctions(QQ).schur(); s[5].plethysm(s[3])
\end{verbatim}
and then selecting modules whose diagrams have at most four parts. We obtain
$$ \begin{matrix}
S_{5, 4, 4, 2}  \oplus S_{6, 4, 4, 1}  \oplus S_{6, 5, 2, 2}  \oplus S_{6, 6, 3}  \oplus S_{7, 4, 2, 2}  \oplus
S_{7, 4, 3, 1}  \oplus S_{7, 4, 4}  \oplus S_{7, 5, 2, 1}  \\ \oplus S_{7, 6, 2}  \oplus S_{8, 3, 2, 2}  \oplus 
S_{8, 4, 2, 1}  \oplus S_{8, 4, 3}  \oplus S_{8, 5, 2}  \oplus S_{8, 6, 1}  \oplus S_{9, 2, 2, 2}  \oplus 
2 S_{9, 4, 2}  \\ \oplus S_{9, 6}  \oplus S_{10, 3, 2}  \oplus S_{10, 4, 1}  \oplus S_{10, 5}  \oplus S_{11, 2, 2}  \oplus
S_{11, 4}  \oplus S_{12, 3}  \oplus S_{13, 2}  \oplus S_{15} .
\end{matrix} $$
The numbers labeling each module are the length of the rows of the Young diagram. 
A highest weight vector analysis 
shows that the quintics are the 36-dimensional module $S_{5,4,4,2} \C^4$. Alternatively, this is the only combination of irreducible modules of dimension 36. \end{proof}

\subsection{Proof of border rank results}

\begin{proposition} \label{alln}
If the border rank and symmetric border rank agree for $r \times r \times r$ tensors of border rank $r$, then they agree for $n \times n \times n$ tensors of border rank $r$, for all $n \geq r$.
\end{proposition}

\begin{proof}
The containment $ \sigma_r(V_n) \subseteq \sigma_r(S_n) \cap L_n $ always holds. It remains to prove the opposite containment.
We use the technique of {\em inheritance} (see~\cite[Example 5.7.3.8 and~\S7.4]{L}). 
Equations for $\sigma_r(S_n)$ consist of $(r+1) \times (r+1)$ minors of flattenings, and copies of equations for $\sigma_r(S_r)$ obtained by choosing a basis of size $r$ in each factor $\C^n$. The $(r+1) \times (r+1)$ minors intersect with $L_n$ to give the minors of the symmetric flattenings, while the equations for $\sigma_r(S_r)$ intersect with $L_n$ to give $\sigma_r(V_r)$ by the hypothesis of the proposition. We can then compare with the equations for $\sigma_r(V_n)$ given in~\cite[Corollary 7.4.2.3]{L}. The equations are the $(r+1) \times (r+1)$ minors of the symmetric flattenings, as well as copies of equations for $\sigma_r(V_r)$ given by choosing the same basis of size $r$ in each factor $\C^n$. All such choices of basis are covered by the non-symmetric choices in the equations for $\sigma_r(S_n)$, hence this proves the reverse containment.
\end{proof}

\begin{proof}[Proof of Theorem~\ref{border}]
 By the Alexander-Hirschowitz theorem~\cite{AH}, the secant variety $\sigma_5(V_4)$ fills the space of symmetric $4 \times 4 \times 4$ tensors.
As in Lemma~\ref{curves}, we compare the equations defining the secant variety $\sigma_r(V_4)$ with the symmetric restriction of the equations defining the non-symmetric secant $\sigma_r(S_4)$, for $1 \leq r \leq 5$.
The result for $r = 1,2,3$ follows from Lemma~\ref{curves} combined with Proposition~\ref{alln}. When $r = 4$ the result follows from Proposition~\ref{salmon}. 
Finally, all tensors outside of $\sigma_4(V_4)$ have symmetric complex border rank five. Proposition~\ref{salmon} implies that they must also have non-symmetric complex border rank five.
\end{proof}

\begin{proof}[Proof of Corollary~\ref{corbor}]

Theorem~\ref{border} combined with Proposition~\ref{alln} shows that all tensors of border rank $r$ also have symmetric border rank $r$, for $1 \leq r \leq 4$. Consider a tensor of symmetric border rank five. Its border rank cannot be four by Theorem~\ref{border}. Hence the border rank is also five.
\end{proof}

\section{Real ranks of cubic surfaces}\label{realranks}

We first prove the following, by combining results from the literature.

\begin{proposition}\label{generic}
Real rank and real symmetric rank coincide for generic real cubic surfaces.
\end{proposition}

\begin{proof}
A generic real cubic surface has complex rank five, 
$ f = l_1^3 + l_2^3 + l_3^3 + l_4^3 + l_5^3 $.
The linear forms $l_i$ define five planes in $\P^3$ that comprise Sylvester's pentahedron. The triple intersections of the planes are the singular points of the Hessian surface. Since $f$ has real coefficients, so does its Hessian surface, and the singular points of the Hessian occur in complex conjugate pairs. Hence the complex linear forms appearing in the decomposition of~$f$ also occur in complex conjugate pairs. There can be zero, one, or two complex conjugate pairs in the decomposition. A cubic $l^3 + \overline{l^3}$, where $l$ is complex and $\overline{l}$ its complex conjugate, has real symmetric rank three. Hence in the first two cases the real symmetric rank is bounded above by six. In~\cite{BBO}, the authors show that the symmetric rank of the third case is also at most six, and therefore that a generic real cubic surface has real symmetric rank five or six. Generic cubic surfaces have flattening rank four, hence we can apply Theorem~\ref{friedland}, which also holds over the field $\R$, to conclude that the real symmetric and non-symmetric ranks coincide up to rank five, and hence up to symmetric rank six.
\end{proof}

We consider special cubic surfaces in more detail, starting with cones over cubic curves.

\begin{proposition}
Real rank and real symmetric rank coincide for cones over cubic curves.
\end{proposition}

\begin{proof}
Such surfaces have flattening rank at most three. We apply a general linear group transformation to obtain a real symmetric tensor with non-zero entries only in its top-left $3 \times 3 \times 3$ block, and we study the cone as a cubic curve. Using Theorem~\ref{friedland}, equality of real rank and real symmetric rank holds whenever the real symmetric rank is at most two more than the flattening rank. To conclude the proof, we check that all real cubic curves have this property~\cite[Table 1]{B}. 
\end{proof}

It remains to consider cubic surfaces of maximal flattening rank and, by Theorem~\ref{friedland}, those of real symmetric rank at least seven.
Among such non-generic cubic surfaces, one family are the reducible surfaces, for which ranges of real ranks are given in~\cite{CGV}. We revisit the reducible real cubic surface from Proposition~\ref{rank7} from the perspective of real rank.

\begin{proposition} The cubic surface $f = x_1 ( x_1 x_2 + x_3^2 + x_4^2)$ has real non-symmetric rank and real symmetric rank seven.
\end{proposition}

\begin{proof}
The surface can be written as
$$ f = -x_1^2 \left( \frac23 x_1 - x_2 + x_3 + x_4 \right) + \frac13 \left( (x_1 + x_3)^3 - x_3^3 + (x_1 + x_4)^3 - x_4^3 \right) .$$
The first term has real symmetric rank three. The remaining terms are expanded as a real symmetric rank four decomposition. This gives an upper bound of seven. Equality follows from the lower bound in Proposition~\ref{rank7}.
\end{proof}

We introduce a tool for obtaining lower bounds on the real non-symmetric rank of a tensor. It is the real analogue to the substitution method in Theorem~\ref{substitution}.

\begin{theorem}[The substitution method over $\R$] \label{rsubstitution}
Let $T \in \R^{n_1} \otimes \R^{n_2}  \otimes \R^{n_3}$ be a tensor of real rank $r$. We can write $T$ in terms of its slices as $T = \sum_{i = 1}^{n_1} e_i \otimes M_i$, where $\{ e_i : 1 \leq i \leq {n_1} \}$ are the elementary basis vectors, and the $M_i$ are $n_2 \times n_3$ real matrices. Reordering indices such that $M_{n_1} \neq 0$, there exist real constants $\lambda_1, \ldots, \lambda_{n_1-1}$ such that the following $(n_1-1) \times n_2 \times n_3$ real tensor has real rank at most $r-1$:
$$ \sum_{i=1}^{n_1-1} e_i \otimes (M_i - \lambda_i M_{n_1}) .$$
If the matrix $M_{n_1}$ has rank one, the real tensor above has real rank exactly $r-1$.
\end{theorem}

\begin{proof}
Assume $T$ has real rank $r$, with real rank decomposition $T = T_1 + \cdots + T_r$. We can express each rank one tensor in the decomposition as $T_k = \sum_{i = 1}^{n_1} \mu_{ki} e_i \otimes L_{k}$ where the $\mu_{ki}$ are real scalars and $L_k$ is a rank one real matrix. 
The slices of $T$ can then be expressed as $M_i = \sum_{k = 1}^r \mu_{ki} L_k$. By the assumption that $M_{n_1}$ is non-zero, we can reorder the terms in the decomposition such that $\mu_{r n_1} \neq 0$. Setting $\lambda_i = \mu_{r i}$, the tensor $\sum_{i=1}^{n_1-1} e_i \otimes (M_i - \lambda_i M_{n_1})$ has all slices expressible as a linear combination of $L_1, \ldots, L_{r-1}$, and hence it has real rank at most $r-1$. The last sentence follows from the fact that if $M_{n_1}$ has rank one, subtracting multiples of it can change the real rank by at most one.
\end{proof}

We illustrate Theorem~\ref{rsubstitution} on the reducible cubic surface $f = x_1(x_1^2 - x_2^2 - x_3^2 - x_4^2)$. It has real symmetric rank seven, and complex rank six~\cite{CGV}. Since the real and complex ranks differ, the usual substitution method in Theorem~\ref{substitution} does not give a tight lower bound on the real rank. We use the real substitution method to bound the real rank below by seven, and hence to conclude that the real rank and real symmetric rank agree.

\begin{proposition}
The cubic surface $f = x_1(x_1^2 - x_2^2 - x_3^2 - x_4^2)$ has real rank and real symmetric rank seven.
\end{proposition}

\begin{proof}
The statement about the symmetric rank is in~\cite{CGV}. For the lower bound on the non-symmetric rank, we use Theorem~\ref{rsubstitution}. For computational convenience we scale the cubic, leaving the rank unchanged, to $x_1(x_1^2 - 3 x_2^2 - 3 x_3^2 - 3 x_4^2)$ or, as a tensor,
$$ e_1 \otimes 
\begin{bmatrix}
1 & 0 & 0 & 0 \\ 
0 & -1 & 0 & 0 \\ 
0 & 0 & -1 & 0 \\ 
0 & 0 & 0 & -1 \\ 
\end{bmatrix}
+ e_2 \otimes
\begin{bmatrix}
0 & -1 & 0 & 0 \\ 
-1 & 0 & 0 & 0 \\ 
0 & 0 & 0 & 0 \\ 
0 & 0 & 0 & 0 \\ 
\end{bmatrix}
+ e_3 \otimes
\begin{bmatrix}
0 & 0 & -1 & 0 \\ 
0 & 0 & 0 & 0 \\ 
-1 & 0 & 0 & 0 \\ 
0 & 0 & 0 & 0 \\ 
\end{bmatrix}
+ e_4 \otimes
\begin{bmatrix}
0 & 0 & 0 & -1 \\
0 & 0 & 0 & 0 \\
0 & 0 & 0 & 0 \\
-1 & 0 & 0 & 0 \\
\end{bmatrix} . $$
We subtract off an arbitrary multiple of the slices of the tensor to give a $4 \times 4 \times 2$, $4 \times 2 \times 2$, and finally a $2 \times 2 \times 2$ tensor. We show that there do not exist {\em real} multiples that can be subtracted to give a tensor of zeros. 
If the pairs of slices we subtract are linearly independent, Theorem~\ref{rsubstitution} then implies that the real non-symmetric rank of $f$ is at least $1 + 2 + 2 + 2 = 7$.

Subtracting off three pairs of slices of $f$ in multiples $s_i, t_i, u_i,  v_i, w_i, x_i$, where $i = 1,2$ denotes which of the two slices we subtract from, gives the $2 \times 2 \times 2$ tensor with slices
\begin{footnotesize}
$$ \begin{bmatrix}  s_1 u_1+t_1 v_1+ t_1 x_1-v_1 x_1+1 & s_2 u_1+ t_2 v_1 + t_2 x_1 - w_1 \\
      s_1 u_2 + t_1 v_2 - v_2 x_1 - w_1 & s_2 u_2 + t_2 v_2-1\end{bmatrix}, \qquad
      \begin{bmatrix} 
      t_1 x_2 - v_1 x_2 + s_1 - u_1 & t_2 x_2 + s_2 - w_2 \\
      -v_2 x_2 - u_2 - w_2 & 0
      \end{bmatrix}.$$\end{footnotesize}We show that the ideal generated by the eight entries does not contain any real points. Eliminating $w_1,w_2,x_1,x_2,t_1,t_2$ gives the hypersurface $(s_2 u_1 + s_1 u_2)^2 + (s_1 - u_1)^2 + (s_2 + u_2)^2 + (s_2 v_1 + u_2 v_1 + s_1 v_2 - u_1 v_2)^2 = 0$.
Over the reals, this if zero if and only if the individual squares in the sum vanish, hence $s_1 = u_1$ and $s_2 = -u_2$. The ideal obtained by eliminating $w_1,w_2,x_1,x_2,v_2$ then has equation $(t_2 u_1 + t_1 u_2)^2 +  (u_1 u_2 - t_2 v_1)^2 + t_1^2 + t_2^2 + u_1^2+u_2^2 = -1$, which has no real solutions. This concludes the main case.

It remains to consider the case when
some pairs of slices of the tensor are linearly dependent. The first and second pairs of slices we subtract are always linearly independent, taking us to a $4 \times 2 \times 2$ tensor whose real rank is four less than that of $f$. The third pair of slices are dependent only if $t_1 = v_1$ and $t_2 = -v_2$. The result then follows as above, by choosing a different pair of slices to subtract, unless $s_1=u_1$ and $s_2 = -u_2$. In this case, the $4 \times 2 \times 2$ tensor has four slices spanned by
$$ M_1 = \begin{bmatrix} s_1 u_1 + t_1 v_1+1 & s_2 u_1 + t_2 v_1 \\\ s_1 u_2 + t_1 v_2 &  s_2 u_2 + t_2 v_2-1 \end{bmatrix} \quad \text{and} \quad M_2 = \begin{bmatrix}0 & 1 \\ 1 & 0 \end{bmatrix} ,$$
with first slice a scalar multiple of $M_1$, and the remaining three slices scalar multiples of $M_2$.
There does not exist a real multiple of $M_2$ that can be subtracted from $M_1$ to give a rank one matrix, because ${\rm det}(M_1 - w_1 M_2) = -(u_2 v_1 -  u_1 v_2)^2 - u_1^2 - u_2^2 - v_1^2 - v_2^2 - w_1^2 -1 = 0$ has no solutions over the reals. By Theorem~\ref{rsubstitution}, the real rank of the $4 \times 2 \times 2$ tensor is at least three. Hence we obtain an overall lower bound of $3 + 2 + 2 = 7$ on the real rank.
\end{proof}

We saw above that every real cubic surface is arbitrarily close to one of real symmetric rank five or six.
The real rank five locus is separated from the real rank six locus by a degree 40 hypersurface~\cite{MM}.
We are interested in the real analogue of Theorem~\ref{border}: to show that the real border rank and real symmetric border rank agree. Generic tensors have the same real rank as real border rank, hence their real border rank and real symmetric border rank agree by Proposition~\ref{generic}. To conclude the paper, we prove the following result.

\begin{proposition}
Real border rank and real symmetric border rank coincide for all cubic surfaces of sub-generic real symmetric border rank.
\end{proposition}

\begin{proof}
The set of real rank one tensors is closed, so we begin by considering a cubic surface of real border rank two. Such cubic surfaces lie in the {\em real rank two locus}, shown in~\cite{SS} to be defined by the non-negativity of the hyperdeterminant of all $2 \times 2 \times 2$ blocks. The locus of real symmetric border rank two tensors is contained in this set, being described by the non-negativity of the diagonal (symmetric) $2 \times 2 \times 2$ blocks~\cite{SS}. All diagonal combinations occur among the non-symmetric inequalities, hence the two sets are equal.

We now consider real border rank three cubic surfaces. Since the flattening rank is bounded above by the border rank, the flattening rank is at most three and we can change coordinates, as in the previous sections, to consider $f$ as a plane cubic curve.
From Theorem~\ref{border}, it suffices to consider the orbits in~\cite[Table 1-2]{B} of cubic curves whose complex (symmetric) border rank is strictly less than their real symmetric border rank. This applies to only one orbit, which has border rank two, hence it is covered by the first paragraph.
Finally, assume $f$ is a cubic surface with real symmetric border rank four. The real non-symmetric border rank cannot be strictly less than four by the above cases. 
\end{proof}

The results in this section constitute progress towards the real rank analogues of Theorem~\ref{rank} and Theorem~\ref{border}. 
Completing the real rank version of Theorem~\ref{rank} requires proving the equality of real rank and real symmetric rank for singular irreducible cubic surfaces, and non-singular cubic surfaces for which Theorem~\ref{sylv} fails to give a decomposition.
To prove Theorem~\ref{border} for real border rank, it remains to consider cubic surfaces whose rank and border rank differ, having real border rank five or six.

We conclude the paper by posing two questions for future study.
The counter-example to Comon's conjecture in~\cite{Sh} is a tensor of size $800 \times 800 \times 800$ and symmetric rank at least $904$. The result in  
Theorem~\ref{rank} gives the agreement of (complex) rank and symmetric rank for all tensors of size $n \times n \times n$ where $n \leq 4$. This suggests the question of finding a tensor of size $n \times n \times n$, with $n$ minimal, whose rank and symmetric rank differ. These results combine to show that $5 \leq n \leq 800$. This is relevant in determining whether rank and symmetric rank agree for the sizes of tensors occurring in a particular application.
Corollary~\ref{corrank} gives the agreement of (complex) rank and symmetric rank for all tensors of symmetric rank at most seven. This suggests the question of finding a tensor of symmetric rank $r$, with $r$ minimal, whose rank and symmetric rank differ. We provide a lower bound of eight while~\cite{Sh} implies an upper bound of 906.

\subsubsection*{Acknowledgements}
I would like to thank Bernd Sturmfels for helpful discussions, and Luke Oeding for sending me a text file of 1728 salmon equations. Thanks also to Yue Ren for help with a Magma computation, and Jon Hauenstein for Bertini assistance.

\begin{small}

\noindent \footnotesize {\bf Author's address:}

\noindent  Anna Seigal,
University of California, Berkeley, USA,
{\tt seigal@berkeley.edu}.
\end{small}

\end{document}